\documentclass[a4paper,12pt,reqno]{amsart}
\usepackage{amssymb}
\usepackage{amsmath}
\usepackage{enumitem}
\usepackage{mathrsfs}
\usepackage[all]{xy}
\usepackage{mathtools}
\usepackage{hyperref}
\usepackage{url}
\usepackage{bbm}

\usepackage[OT2,T1]{fontenc}
\DeclareSymbolFont{cyrletters}{OT2}{wncyr}{m}{n}

\usepackage[margin=1.3in]{geometry}
\numberwithin{equation}{section} \numberwithin{figure}{section}

\DeclareMathOperator{\Pic}{Pic} 
 
\DeclareMathOperator{\Gal}{Gal} 
 
\DeclareMathOperator{\Spec}{Spec}
 
\DeclareMathOperator{\Hom}{Hom} 
\DeclareMathOperator{\im}{Im}

\DeclareMathOperator{\Br}{Br} 

\DeclareMathOperator{\inv}{inv}

\DeclareMathOperator{\HH}{H} 
\newcommand{\R}{\mathrm{R}} 
\let\div\relax
\DeclareMathOperator{\div}{div}
\DeclareMathOperator{\diag}{diag}

\DeclareSymbolFont{cyrletters}{OT2}{wncyr}{m}{n}
\DeclareMathSymbol{\Sha}{\mathalpha}{cyrletters}{"58}
\DeclareMathSymbol{\Be}{\mathalpha}{cyrletters}{"42}
\newcommand{\mmu}{\boldsymbol{\mu}}

\newcommand{\OO}{\mathcal{O}}

\newcommand\FF{\mathbb{F}}
\newcommand\PP{\mathbb{P}}
\renewcommand\AA{\mathbb{A}}
\newcommand\ZZ{\mathbb{Z}}
\newcommand\NN{\mathbb{N}}
\newcommand\QQ{\mathbb{Q}}
\newcommand\RR{\mathbb{R}}
\newcommand\CC{\mathbb{C}}
\newcommand\GG{\mathbb{G}}

\newcommand\Gm{\GG_\mathrm{m}}
\newcommand{\Adele}{\mathbf{A}}

\newcommand{\sU}{\mathcal{U}}

\newcommand{\bu}{{\bf u}}

\newtheorem{lemma}{Lemma}

\newtheorem{theorem}[lemma]{Theorem}
\newtheorem{proposition}[lemma]{Proposition}
\newtheorem{corollary}[lemma]{Corollary}

\theoremstyle{definition}

\newtheorem{remark}[lemma]{Remark}

\numberwithin{lemma}{section}

\usepackage[usenames, dvipsnames]{color}
\newcommand{\dan}[1]{{\color{blue} \sf $\clubsuit\clubsuit\clubsuit$ Dan: [#1]}}
\newcommand{\martin}[1]{{\color{ForestGreen} \sf $\spadesuit\spadesuit\spadesuit$ Martin: [#1]}}

\begin{document}
\title{Brauer--Manin obstruction for Erd\H{o}s--Straus surfaces}
\author{Martin Bright}
\address{Mathematisch Instituut \\ Niels Bohrweg 1 \\ 2333 CA Leiden \\ Netherlands}
\email{m.j.bright@math.leidenuniv.nl}

\author{Daniel Loughran}
  \address{
	Department of Mathematical Sciences\\
	University of Bath\\
	Claverton Down\\
	Bath\\
	BA2 7AY\\
	UK}
\urladdr{https://sites.google.com/site/danielloughran/}


\begin{abstract}
	We study the failure of the integral Hasse principle and strong approximation for the Erd\H{o}s--Straus conjecture using the Brauer--Manin 
	obstruction.
\end{abstract}

\maketitle
\thispagestyle{empty}
\tableofcontents

\section{Introduction}

\subsection{The Erd\H{o}s--Straus conjecture}

The Erd\H{o}s--Straus conjecture states that for every $n \geq 2$
the equation
\begin{equation} \label{eqn:ES}
	\frac{4}{n} = \frac{1}{u_1} + \frac{1}{u_2} + \frac{1}{u_3}
\end{equation}
always has a solution with $u_1,u_2,u_3 \in \NN$. Note that there is always
a solution with $u_1,u_2,u_3 \in \ZZ$ \cite{Jar04}, and to prove the conjecture it suffices to consider the case where $n$ is a prime. 
Moreover for any fixed $n$, it is straightforward to see that there can be only finitely many solutions, and that they may be easily enumerated (see Lemma~\ref{lem:Zariski}).
We refer to Mordell's book \cite[Ch.~30]{Mor69} and the more recent paper \cite{ET13} for further background and history on this problem. 

In this paper we investigate what modern techniques from arithmetic geometry can say about this conjecture and more generally the structure of the solutions to \eqref{eqn:ES}. 
At a first glance, it is not clear how to use tools from modern algebraic geometry to tackle the problem, as $\NN$ is not a ring. 
However, this conjecture does indeed have a natural interpretation as a question of \emph{strong approximation}, stipulating that integer solutions with certain real conditions exist.
Our first main result states that there is no Brauer--Manin obstruction in this case (see \S \ref{sec:geometry_intro} for a more precise statement and background on the Brauer--Manin obstruction).

\begin{theorem}	 \label{thm:No_Br}
	Let $n \geq 2$. Then there is no Brauer--Manin obstruction
	to the existence of natural number solutions of the equation \eqref{eqn:ES}.
\end{theorem}

Despite there being no Brauer--Manin obstruction to the conjecture, it turns out that there is in fact an obstruction to strong approximation at the $p$-adic places. This obstruction has the following completely explicit description. (In the statement $(\cdot, \cdot)_p$ denotes the Hilbert symbol.)

\begin{theorem} \label{thm:Hilbert}
	Let $n \in \NN$ be odd  and $\bu \in \NN^3$ a solution to 
	\eqref{eqn:ES}. Then
	\[ \prod_{p \mid n} (-u_1/u_3, -u_2/u_3)_p = -1. \]
\end{theorem}

Despite the apparent asymmetry, the given Hilbert symbols are actually invariant under the natural action of the symmetric group $S_3$ on the variables $u_i$ (see Proposition \ref{prop:descend}).
In the stated generality, Theorem \ref{thm:Hilbert} does not seem to have been known and gives new conditions which natural number solutions must satisfy. Theorem \ref{thm:Hilbert} allows one to recover various known results in a more systematic and conceptual way, as special cases of a Brauer--Manin obstruction. For example if $n$ is an odd prime, we have the following.

\begin{corollary} \label{cor:Yamamoto}
	Let $n=p$ be an odd prime and $\bu \in \NN^3$ a solution to 
	\eqref{eqn:ES}.
	Then there exists $i \neq j$ such that $u_i/u_j \in \ZZ_p^*$.
	For such a solution we have
	\[ \left(\frac{-u_i/u_j}{p}\right) = -1, \]
	where the symbol is the Legendre symbol.
\end{corollary}

Corollary \ref{cor:Yamamoto} unifies various quadratic reciprocity conditions found by Yamamoto \cite{Yam65} for $p \equiv 1 \bmod 4$. We are also able to recover the following result of Elsholtz and Tao (\cite[Prop.~1.6]{ET13}).

\begin{corollary} \label{cor:square}
	If $n$ is an odd square, then there are no natural
	number solutions $\bu$ with
	\[n \mid u_1, \gcd(n,u_2u_3) = 1, \quad \mbox{ or }
	\quad 	\gcd(n,u_1) = 1, n \mid u_2, n \mid u_3. \]
\end{corollary}

Corollary \ref{cor:square} is really a condition on natural number solutions which is not present for integer solutions (e.g.~for $n=9$ consider the solutions $(-18, 4, 4)$ and $(-9, 2, 18)$).
Similarly, the congruence condition in Corollary \ref{cor:Yamamoto} is also not present for integer solutions in general. For example, consider  $p = 5$ and the solution $(-5,2,2)$, where the corresponding Legendre symbol is
$1$. In fact, for integer solutions which are not natural number solutions,
the exact opposite of Theorem \ref{thm:Hilbert} holds.

\begin{theorem} \label{thm:Hilbert2}
	Let $n$ be an odd integer and $\bu \in \ZZ^3$ a solution to 
	\eqref{eqn:ES} which is not a natural number solution. Then
	\[\prod_{p \mid n} (-u_1/u_3, -u_2/u_3)_p = 1. \]
\end{theorem}

\subsection{Geometric interpretation} \label{sec:geometry_intro}
We now explain in more detail how to interpret our results geometrically using the Brauer--Manin obstruction.
Consider the corresponding algebraic surface derived from \eqref{eqn:ES}
\begin{equation} \label{def:U_n}
	U_n: \quad 4u_1u_2u_3 = n(u_1u_2 + u_1u_3 + u_2u_3) \subset \mathbb{A}_\QQ^3.
\end{equation}
This is an affine cubic surface, and geometrically a so-called \emph{log K$3$ surface}. Many interesting classical Diophantine equations turn out to concern log K3 surfaces, and their integer points are an active area of research \cite{CTW12, CTWX, Har17, Har19, JS17, LM19}. Note that $U_n$ is singular, with the unique singular point lying at the origin. 

We let $\mathcal{U}_n$ denote the natural model for $U_n$ given by the same equation in $\mathbb{A}^3_{\ZZ}$. Note that $U_1 \cong U_n$ over $\QQ$ for all $n \in \NN$, by simply rescaling the $u_i$. The Erd\H{o}s--Straus conjecture therefore concerns existence of certain integer points on different models over $\ZZ$ of the same surface over $\QQ$; in particular this nicely highlights the fact that  different models of the same surface can give rise to very different problems in general.

Let $\pi_0(U_n(\RR))$ be the set of connected components of $U_n(\RR)$ and $\Adele_{\QQ,f}$ the ring of finite adeles.
One says that $U_n$ \emph{satisfies strong approximation} if $U_n(\QQ)$ has dense image in  $
\sU_n(\Adele_\QQ)_{\bullet}:=\pi_0(U_n(\RR)) \times \sU_n(\Adele_{\QQ,f});
$
equivalently, if 
\begin{equation} \label{eqn:W}
	U_n(\QQ) \cap W \neq \emptyset
\end{equation}
for all non-empty open subsets $W \subset \sU_n(\Adele_\QQ)_{\bullet}$. We work with $\sU_n(\Adele_\QQ)_{\bullet}$ since $U_n(\QQ) \subset \sU_n(\Adele_\QQ)$ is discrete as $U_n$ is affine, hence clearly not dense.
We let
\[U_n(\RR)_+ = \{ \mathbf{u} \in U_n(\RR) : u_1,u_2,u_3 > 0\}. \]
We will show that $U_n(\RR)_+$ is a connected component of $U_n(\RR)$, and its complement is also a connected component. We define $\sU_n(\NN) := \sU_n(\ZZ) \cap U_n(\RR)_+$. The Erd\H{o}s--Straus conjecture  is equivalent to \eqref{eqn:W} for $W = \{U_n(\RR)_+\}\times \prod_p \sU_n(\ZZ_p)$, hence stipulates that a special case of strong approximation holds. One can even formulate the conjecture as a problem of strong approximation for $U_1$; here it is equivalent to \eqref{eqn:W} for $U_1$ and $W_n$ for all $n \geq 2$, where
\[W_n = \{U_1(\RR)_+\}\times \prod_{p \mid n}\{ \bu_p \in  U_1(\QQ_p) : v_p(u_i) \leq -v_p(n) \mbox{ for all } i \} \times \prod_{p \nmid n} \sU_1(\ZZ_p). \]

We now recall how one can use the Brauer group to study this problem (see \cite[\S8.2]{Poo17} for further background on the Brauer--Manin obstruction). Recall that there is a right continuous pairing
\[\Br U_n \times \sU_n(\Adele_\QQ)_{\bullet} \to \QQ/\ZZ \]
given by pairing with an element of $\Br U_n$ and taking the sum of local invariants. For an open subset $W \subset \sU_n(\Adele_\QQ)_{\bullet}$, we define $W^{\Br}$ to be the right kernel of this pairing restricted to $W$. We have $U_n(\QQ) \cap W \subset W^{\Br}$; in particular, if $W^{\Br} = \emptyset$ then $U_n(\QQ) \cap W = \emptyset$ and one says that there is a Brauer--Manin obstruction to strong approximation (cf.~\eqref{eqn:W}).
We first calculate the Brauer group.

\begin{theorem} \label{thm:Brauer}
	We have
	\[\Br U_n / \Br \QQ \cong \ZZ/2\ZZ \]
	generated by the quaternion algebra $(-u_1/u_3,-u_2/u_3)$.
\end{theorem}

The algebra in Theorem \ref{thm:Brauer} is transcendental, meaning that it does not become trivial after base change to an algebraic closure of $\QQ$, so we will obtain new cases of a transcendental Brauer--Manin obstruction on log K3 surfaces.
One novel feature is that there are few examples in the literature where Brauer groups of singular varieties have been computed, as Brauer group computations usually use Grothendieck's purity theorem which requires regularity (or at least a singular locus of large codimension). We prove Theorem \ref{thm:Brauer} by first calculating the Brauer group of a desingularisation, then showing that every such Brauer group element comes from the singular surface. 

This latter property is a special case of a more general result about Brauer groups of singular surfaces, which may be of independent interest and does not seem to have been noticed before. Recall that a normal variety $Y/k$ is said to have only \emph{rational singularities} if there exists a desingularisation $\widetilde{Y} \to Y$ for which all the higher direct images of $\OO_{\widetilde{Y}}$ are trivial.

\begin{theorem} \label{thm:rational}
	Let $U$ be a normal surface over a field $k$ of characteristic $0$
	with rational singularities and 
	$f \colon \widetilde{U} \to U$ a desingularisation.
	Then the induced map
	$\Br U \to \Br \widetilde{U}$
	is surjective.	
\end{theorem}

One could hope to use the Brauer group element from Theorem \ref{thm:Brauer} to disprove the Erd\H{o}s--Straus conjecture by showing that $(U_n(\RR)_{+} \times \prod_p \sU_n(\ZZ_p))^{\mathrm{Br}} = \emptyset$; our next result says that this does not happen.

\begin{theorem} \label{thm:BM}
	For all $n \in \NN$ we have
	\begin{align} 
		&(U_n(\RR)_{+} \times \prod_p \sU_n(\ZZ_p))^{\mathrm{Br}} \neq \emptyset \label{eqn:BR},\\
		&(U_n(\RR)_{+} \times \prod_p \sU_n(\ZZ_p))^{\mathrm{Br}} \neq U_n(\RR)_{+} \times \prod_p \sU_n(\ZZ_p). \label{eqn:BR_SA}
	\end{align}
\end{theorem}
The first equation \eqref{eqn:BR} is a more precise version of Theorem  \ref{thm:No_Br}. The second equation \eqref{eqn:BR_SA} says that nonetheless there is always a Brauer--Manin obstruction to strong approximation for natural number solutions (as manifested by  Theorems \ref{thm:Hilbert} and \ref{thm:Hilbert2}).

Despite there being a Brauer--Manin obstruction to strong approximation, it turns out that not every failure of strong approximation is explained by the Brauer--Manin obstruction.

\begin{theorem} \label{thm:Wei}
	For all $n \in \NN$, the map
	\[ U_n(\QQ) \to \sU_n(\AA_\QQ)^{\Br}_\bullet \]
	does not have dense image.
\end{theorem}

We prove this by showing that $\sU_n(\ZZ)$ is not Zariski dense using real considerations. The conclusion then follows from the fact that $\Br U_n / \Br \QQ$ is finite.

\begin{remark}
	In this paper we focus on the original conjecture of 
	Erd\H{o}s--Straus concerning the equation \eqref{eqn:ES}.
	A more general conjecture, due to Schinzel \cite{Sie56},
	states that given $m \geq 3$, 
	for all $n > n_0(m)$ there exists $u_i \in \NN$ such that
	\[\frac{m}{n} = \frac{1}{u_1} + \frac{1}{u_2} + \frac{1}{u_3}. \]
	These surfaces are again $\QQ$-isomorphic, hence
	Theorem \ref{thm:Brauer} still holds here.
	A minor adaptation of our method shows the following
	analogue of Theorem \ref{thm:Hilbert}: for all solutions with
	$\bu \in \NN^3$ we have
	\[\prod_{p \mid 2nm} (-u_1/u_3, -u_2/u_3)_p = -1. \]
	Moreover versions of Theorems \ref{thm:Hilbert2}, \ref{thm:BM}, and \ref{thm:Wei} also hold in this case.
\end{remark}

\subsection*{Outline of the paper}
In \S 2 we study the geometry of Erd\H{o}s-Straus surfaces over a field $k$ of characteristic zero. We calculate the desingularisation, the Picard group, and the Brauer group (Theorem \ref{thm:Brauer}). In \S 3 we apply our knowledge of the Brauer group to prove the remaining results from the introduction. The appendix explains in more detail how Corollary \ref{cor:Yamamoto} relates to results of Yamamoto \cite{Yam65}.

\subsection*{Notation}
For a field $k$, we denote by $\mu(k)$ the group of roots of unity in $k$. For a scheme $X$, we denote by $\Br X = \HH^2(X, \Gm)$ its (cohomological) Brauer group.

\subsection*{Acknowledgements}
We thank Yang Cao, Jean-Louis Colliot-Th\'{e}l\`{e}ne, and Christian Elsholtz for useful comments and references. 
This work was undertaken at the Institut Henri Poincar\'e during the trimester ``Reinventing rational points''. The authors thank the organisers and staff for ideal working conditions. We are grateful to the referee for numerous helpful comments. The second-named author is supported by EPSRC grant EP/R021422/2.

\section{Geometry of Erd\H{o}s--Straus surfaces} \label{sec:Geometry}

In this section we study the geometry of the surfaces $U_n$ from \eqref{def:U_n}. We work over a field $k$ of characteristic $0$ with algebraic closure $\bar{k}$. The primary aim of this section is to prove Theorem \ref{thm:Brauer}. We also prove a result of independent interest on Brauer groups of rational surface singularities (Theorem \ref{thm:rational}).

\subsection{The Cayley cubic and its lines} \label{sec:Cayley}
We let 
\[S_n: \quad 4x_1x_2x_3 = n(x_0x_1x_2 + x_0x_1x_3 + x_0x_2x_3). \]
 be the closure of $U_n$ in $\PP^3_k$, with $U_n$ being the affine patch $x_0 \neq 0$ with variables $u_1,u_2,u_3$.
For $n=-4$, this projective surface is known as  Cayley's (nodal) cubic surface; 
every $S_n$ is isomorphic over $k$ to the Cayley cubic surface. The surface $S_n$ has $4$ singularities, each of type $\mathbf{A}_1$, given by setting all but one coordinate equal to $0$; we let $P= (1:0:0:0)$ be the singularity in $U_n$. The Cayley cubic has $9$ lines over $\bar{k}$. This induces $6$ lines  on $U_n$, of which we are interested
 in the following $3$ lines
\[
	L_{i,j}: u_i = u_j = 0, \quad i \neq j \in \{1,2,3\}.
\]
\subsection{Desingularisation} \label{sec:desing}
Let $\widetilde{U}_n$ be the desingularisation of $U_n$ given by blowing up $P$ once, with exceptional curve $E \subset \widetilde{U}_n$. By abuse of notation, we denote by $L_{i,j}$ the strict transform of the  relevant lines in $\widetilde{U}_n$. We have the equation
\[\widetilde{U}_n: \quad 
4 u_1y_2y_3 = n(y_1y_2 + y_1y_3 + y_2y_3), \,\,
y_iu_j = y_ju_i, \,\, i, j \in \{1,2,3\} \quad \subset \mathbb{A}^3 \times \PP^2, \]
where $u_1,u_2,u_3$ are coordinates on $\mathbb{A}^3$, and $y_1,y_2,y_3$ are homogeneous coordinates on $\PP^2$.
With respect to this equation, the curves of interest to us are
\[E: u_1=u_2=u_3=0, \quad L_{i,j}:y_i = y_j = 0, \, i \neq j \in \{1,2,3\}. \]
One checks that
\begin{equation} \label{eqn:div}
	\frac{u_i}{u_j}= \frac{y_i}{y_j}, \quad \div \frac{y_1}{y_3} = L_{1,2} - L_{2,3}, \quad 	\div \frac{y_2}{y_3} = L_{1,2} - L_{1,3}.
\end{equation}


\subsection{Parametrisation}
Any cubic surface with a rational singularity is rational, with a birational parametrisation given by projecting away for the singular point. Applying this to the singularity $P$, we obtain the birational map to $\PP^2$.
On the desingularisation, this  becomes the birational morphism 
\begin{equation} \label{eqn:param}
	\widetilde{U}_n \to \PP^2, \quad (u_1,u_2,u_3; y_1:y_2:y_3) \mapsto (y_1:y_2:y_3).
\end{equation}
We let
\begin{equation} \label{def:V}
	V_n:=  \widetilde{U}_n \setminus \{y_1y_2y_3 = 0\}.
\end{equation}
Note that the boundary is the disjoint union of the lines $L_{i,j}$
\begin{equation} \label{eqn:disjoint}
\widetilde{U}_n \setminus V_n = L_{1,2} \sqcup L_{2,3} \sqcup L_{3,1}.
\end{equation}
The following important observation will be used numerous times.

\begin{lemma} \label{lem:Gm}
	We have $V_n \cong \Gm^2$ and $\HH^0(V_{n,\bar{k}}, \Gm) 
	\cong 
	\bar{k}^* \bigoplus \ZZ^2$,  with the $\ZZ^2$ factor generated by $y_1/y_3$ and $y_2/y_3$.
\end{lemma}
\begin{proof}
	That $V_n \cong \Gm^2$ follows from the fact that
	the map \eqref{eqn:param} becomes an isomorphism onto 
	its image when restricted to $V_n$. The second part follows from
	the fact that the invertible regular functions on $\Gm^2$
	are generated by characters and non-zero constants.
\end{proof}

\begin{lemma} \label{lem:constants}
	$\HH^0(\widetilde{U}_{n,\bar{k}}, \Gm) = \bar{k}^*.$
\end{lemma}
\begin{proof}
	By Lemma \ref{lem:Gm} any invertible regular 
	function must be a non-trivial 
	product of powers of $y_1/y_3$ and $y_2/y_3$, modulo constants.
	However, such a function cannot be invertible on $\widetilde{U}_n$
	since its divisor is always non-trivial by \eqref{eqn:div}.
\end{proof}

\subsection{Picard group}

\begin{lemma} \label{lem:Pic}
	We have $\Pic \widetilde{U}_n = \Pic \widetilde{U}_{n, \bar{k}} \cong \ZZ$
	generated by $L_{1,2}$.
\end{lemma}
\begin{proof}
	By Lemma \ref{lem:Gm} and \eqref{eqn:div}
	we have the exact sequence
	\[
		0 \to \langle y_1/y_3, y_2/y_3 \rangle \to 
		\langle L_{1,2}, L_{2,3}, L_{3,1} \rangle \to \Pic \widetilde{U}_{n,\bar{k}} \to \Pic V_{n,\bar{k}} \to 0,
	\]
	where the second map associates to a rational function its divisor
	and the third map associates to a divisor its class.
	But $\Pic V_{n,\bar{k}} = 0$ by Lemma \ref{lem:Gm}. 
	The result now follows from \eqref{eqn:div}.
\end{proof}

\subsection{Brauer group}

\subsubsection{Brauer group of $\widetilde{U}_n$}
We denote by $\Br_1 X = \ker(\Br X \to \Br X_{\bar{k}})$
the algebraic Brauer group of a variety $X/k$.

\begin{lemma}
	$\Br_1 \widetilde{U}_{n}  = \Br k.$
\end{lemma}
\begin{proof}
	Lemma \ref{lem:constants} and the Hochschild--Serre spectral sequence
	give an injection $\Br_1 \widetilde{U}_n/\Br k \hookrightarrow
	 \HH^1(k, \Pic \widetilde{U}_{n,\bar{k}})$.
	But $\Pic \widetilde{U}_{n,\bar{k}} = \ZZ$ with trivial Galois action
	by Lemma \ref{lem:Pic}, hence
	this Galois cohomology group is trivial.
\end{proof}

We now find the Galois action on the Brauer group. We denote by 	$\QQ/\ZZ(-1):=\Hom(\mu(\bar{k}), \QQ/\ZZ)$, and refer to \cite[\S2.5]{GS} for background on cyclic algebras.

\begin{proposition} \label{prop:Br_Galois}
	The natural map
	$ \Br \widetilde{U}_{n,\bar{k}} \to \Br V_{n,\bar{k}},$
	induced by the inclusion $V_n \subset \widetilde{U}_n$,
	is an isomorphism. 	In particular 
	$\Br \widetilde{U}_{n,\bar{k}} \cong
	\QQ/\ZZ(-1)$
	as a Galois module, and its elements are represented 
	by the cyclic algebras
	\[(u_1/u_3, u_2/u_3)_\zeta, \quad \zeta \in \mu(\bar{k}).\]
\end{proposition}
\begin{proof}
	The explicit description of $\Br V_{n,\bar{k}}$ follows from 
	Lemma \ref{lem:Gm} and the fact that 
	$\Br \mathbb{G}_{m,\bar{k}}^2 \cong \QQ/\ZZ(-1)$, 
	given by the stated cyclic algebras (see \cite[\S8.1]{CTS20}
	-- note that for $\sigma \in \Gal(\bar{k}/k)$ we have
	$\sigma(\alpha_\zeta) = \alpha_{\sigma(\zeta)}$, but 
	for $\zeta$ an $n$th root of unity and 
	 $a \in (\ZZ/n\ZZ)^*$ we have $a \alpha_\zeta = \alpha_{\zeta^{a^{-1}}}$).

	So let $b = (u_1/u_3, u_2/u_3)_\zeta$. It suffices to show
	that $b$ is unramified along the boundary \eqref{eqn:disjoint}.
	The $L_{i,j}$ are regular
	and disjoint, hence Grothendieck's purity theorem \cite[Cor.~6.2]{grot}
	yields the exact sequence
	\[0 \to \Br \widetilde{U}_{n,\bar{k}} \to \Br V_{n,\bar{k}}
	\to \bigoplus_{i \neq j} \HH^1(L_{i,j,\bar{k}},\QQ/\ZZ)\]
	where the last map is the residues along the $L_{i,j,\bar{k}}$.
	(Note that the hypothesis that the boundary divisor be regular is missing from Grothendieck's statement, but it holds in our case.)
	However $L_{i,j,\bar{k}} \cong \mathbb{A}^1_{\bar{k}}$
	is simply connected,
    so the corresponding residues are trivial. The result follows.
\end{proof}

We next show that every Galois-invariant element of $\Br \widetilde{U}_{n,\bar{k}}$ in fact descends to the ground field $k$. To do this, we make use of the relation
\begin{equation} \label{eqn:awesome}
	-\frac{u_i}{u_j} =
	\frac{1}{1 + u_j/u_k - 4u_j/n}, \quad \{i,j,k\} = \{1,2,3\},
\end{equation}
derived from \eqref{eqn:ES}. (This relation will also appear in other parts of the paper).

\begin{proposition} \label{prop:descend}
	The natural map 
	$ \Br \widetilde{U}_{n} \to (\Br \widetilde{U}_{n,\bar{k}})^{\Gal(\bar{k}/k)}$
	is surjective. A complete set of representatives for 
	the elements of $\Br \widetilde{U}_{n}/ \Br k$ is given
	by the cyclic algebras
	\[\alpha_\zeta = (-u_1/u_3,-u_2/u_3)_\zeta, \quad \zeta \in \mu(k).\]
	These algebras have the following equivalent representations:
	\begin{align*}
	\alpha_\zeta & = (-u_i/u_k, -u_j/u_k)_\zeta  = (-y_i/y_k, -y_j/y_k)_\zeta,
	\quad \{i,j,k\} = \{1,2,3\}.
	\end{align*}
\end{proposition}
\begin{proof}
	By Proposition \ref{prop:Br_Galois}, we have 
	$(\Br \widetilde{U}_{n,\bar{k}})^{\Gal(\bar{k}/k)}
	\cong (\QQ/\ZZ(-1))^{\Gal(\bar{k}/k)}$, and this
	 is (non-canonically) isomorphic to $\mu(k)$ \cite[Lem.~2.4]{CTWX}.
	By Proposition \ref{prop:Br_Galois}, the cyclic algebras 
	$\alpha_\zeta$ therefore give a complete set of representatives
	for the Galois-invariant elements. It thus suffices to show that
	these descend to $k$.

	The different representations are	
	easily checked to hold 	in the Brauer group of 
	the function	field of $U_n$,
	using \eqref{eqn:div} and
	the relation $(a,b)_\zeta = (-b/a,1/a)_\zeta$. 
	To show that $\alpha_\zeta$ is 
	unramified along the $L_{i,j}$, we use~\eqref{eqn:div}.
	By symmetry, it suffices to show that $\alpha_\zeta$ is unramified along 
		$L_{2,3}$. However, by \eqref{eqn:div} and
	standard formulae for residues~\cite[Prop.~7.5.1, Ex.~7.1.5]{GS},
	the residue of $\alpha_\zeta$ along $L_{2,3}$ is
	\[-u_2/u_3 \in k(L_{2,3})^{*}/(k(L_{2,3})^{*})^m,\]
	where $m$ is the order of $\zeta$.
	But using the relation \eqref{eqn:awesome}, we have
	\[-\frac{u_2}{u_3} = \frac{1}{1 + u_3/u_1 - 4u_3/n},\]
	so that the residue is in fact equal to $1$ along $L_{2,3}$
	as $u_3 = 0$ here. This shows that $\alpha_\zeta \in \Br \widetilde{U}_n$,
	as required.
\end{proof}

Note that Proposition \ref{prop:descend} shows that $\Br \widetilde{U}_n/\Br k$ is finite if $k$ is a number field; something which is not a priori obvious.


\begin{corollary} \label{cor:Z/2Z}
	If $k = \QQ$, then
	$\Br \widetilde{U}_n/ \Br \QQ$ is isomorphic to $\ZZ/2\ZZ$
	generated by the class of the quaternion algebra 
	\[\alpha := \alpha_{-1}= (-u_1/u_3,-u_2/u_3).\]
\end{corollary}

\begin{remark}
	Note that the ``obvious'' Galois-invariant element 
	$(u_1/u_3,u_2/u_3)$ does not descend to $\QQ$. Despite
	being unramified over $\bar{\QQ}$, it ramifies over the lines
	$L_{i,j}$ with constant (non-trivial) residue. 
	We have multiplied this element by some ramified \emph{algebraic} Brauer
	group elements to kill these constant residues.
\end{remark}

\subsubsection{Brauer group of $U_n$}

We  calculated the Brauer group of the desingularisation $\widetilde{U}_n$ using Grothendieck's purity theorem. This method uses that $\widetilde{U}_n$ is smooth and does not apply directly to $U_n$. To calculate $\Br U_n$ we shall use Theorem \ref{thm:rational}, which we now prove. 

\begin{proof}[Proof of Theorem \ref{thm:rational}]
	We compute the higher direct images $\R^q f_* \Gm$ with respect to the \'etale topology and use the Leray spectral sequence for the morphism $f$ and the sheaf $\Gm$; the necessary material can be found in~\cite[\S III, \S IV]{lipman}.

	Let $P_1, \dotsc, P_r$ be the closed points at which $U$ is singular,
	with residue fields $k_j=\kappa(P_j)$, and let $E_j$
	be the exceptional divisor above $P_j$.
	Let $\bar{P}_j$ be a geometric point  above $P_j$, and let $\bar{E}_j$ the fibre above $\bar{P}_j$.
	By~\cite[Prop.~1]{artin}, $\bar{E}_j$ is a tree of $\PP^1$s.
	By~\cite[Prop.~11.1]{lipman}, $\Pic \bar{E}_j$ is isomorphic to $\ZZ^{d_j}$, where $d_j$ is the number of irreducible components of $\bar{E}_j$, with the absolute Galois group of $k_j$ permuting the factors as it permutes the irreducible components.
	
	Let $\OO_{U,P_j}^\textrm{sh}$ be a strict Henselisation of the local ring of $U$ at $P_j$.  
	The standard calculation of the stalks of higher direct images shows that $(\R^1 f_* \Gm)_{\bar{P}_j}$ is isomorphic to $\Pic(\widetilde{U} \times_U \Spec \OO_{U,P_j}^\textrm{sh})$.
	The natural map $\Pic(\widetilde{U} \times_U \Spec \OO_{U,P_j}^\textrm{sh}) \to \Pic \bar{E}_j$ is injective by~\cite[Thm.~12.1]{lipman} and surjective by~\cite[Lem.~14.3]{lipman}, so is an isomorphism.
	We deduce that  $(\R^1 f_* \Gm)_{\bar{P}_j}$ and $\Pic \bar{E}_j$ are isomorphic as Galois modules over $k_j$.
	Let $i_j: P_j \to U$ be the inclusion.  Given that $\R^1 f_* \Gm$ is supported at the points $P_j$, we have computed
	\begin{equation}\label{eq:R1}
	\R^1 f_* \Gm \cong \prod_j (i_j)_* \Pic \bar{E}_j.
	\end{equation}
	It follows that
	\begin{equation}\label{eq:H1R1}
	\HH^1(U, \R^1 f_* \Gm) = \prod_j \HH^1(k_j, \Pic \bar{E}_j) = 0,
	\end{equation}
	since $\Pic \bar{E}_j$ is an induced module.

	We now show that the stalks $(\R^2 f_* \Gm)_{\bar{P}_j}$ are torsion-free.
	The Kummer sequence on $\widetilde{U}$ gives, for any $m \ge 1$, an exact sequence
	\[
	\R^1 f_* \Gm \xrightarrow{\times m} \R^1 f_* \Gm \to
	\R^2 f_* \mmu_m \to \R^2 f_* \Gm \xrightarrow{\times m} \R^2 f_* \Gm.
	\]
	Proper base change \cite[Cor.~VI.2.7]{milne} shows
	\[
	(\R^2 f_* \mmu_m)_{\bar{P}_j} \cong \HH^2(\bar{E}_j, \mmu_m) \cong \Pic \bar{E}_j/m\Pic \bar{E}_j
	\] where the last isomorphism follows from the 
	Kummer sequence of $\bar{E}_j$, as $\Br \bar{E}_j = 0$ 	by
	\cite[Cor.~1.2]{grot}. 
	Therefore  $(\R^1 f_* \Gm)_{\bar{P}_j}$ surjects onto $(\R^2 f_* \mmu_m)_{\bar{P}_j}$ by \eqref{eq:R1}, showing
	 that $(\R^2 f_* \Gm)_{\bar{P}}$ has no non-trivial $m$-torsion.
	
	Using~\eqref{eq:R1} and~\eqref{eq:H1R1}, the Leray spectral sequence for the morphism $f$ and the sheaf $\Gm$ now gives an exact sequence
	\begin{equation} \label{eqn:Leray}
	\Pic U \to \Pic \widetilde{U} \to \prod_j \HH^0(k_j,\Pic \bar{E}_j)
	\to \Br U \to \Br \widetilde{U} \to \HH^0(U,\R^2 f_* \Gm).
	\end{equation}
	Since $\widetilde{U}$ is regular, $\Br \widetilde{U}$ is a subgroup of $\Br k(\widetilde{U})$ and is therefore torsion. 
	Thus the rightmost arrow is zero. 
	This proves that $\Br U \to \Br \widetilde{U}$ is surjective.
\end{proof}

In the case of Erd\H{o}s--Straus surfaces, we obtain the following stronger result.

\begin{corollary} \label{cor:Br}
	The natural map
	$\Br U_n \to \Br \widetilde{U}_n$
	is an isomorphism.
\end{corollary}
\begin{proof}

	By Theorem \ref{thm:rational}, it suffices to show
	that the stated map is injective. The exact sequence \eqref{eqn:Leray}
	here reads 
	\[
	\Pic U_n \to \Pic \widetilde{U}_n \to \Pic E
	\to \Br U_n \to \Br \widetilde{U}_n \to 0.
	\]
	But $\Pic \widetilde{U}_n \to \Pic E$ is surjective as the strict transform of $L_{1,2}$ has intersection number $1$ with the exceptional divisor $E$. This completes the proof.
\end{proof}

Corollaries \ref{cor:Z/2Z} and \ref{cor:Br} in particular prove Theorem \ref{thm:Brauer}.

\begin{remark}
	The map in Theorem \ref{thm:rational} need not be an isomorphism
	in general. If $X$ is the Cayley cubic surface in $\PP^3_\CC$,
	then $\Br X \cong \ZZ/2\ZZ$ \cite[Tab.~2]{Bri13},
	but the Brauer group of the desingularisation
	is clearly trivial.
\end{remark}

\begin{remark} \label{rem:smooth_points}
	We have calculated $\Br U_n$ for completeness; however, we could
	just have chosen to  work on the desingularisation instead.
	Namely, consider the Brauer group element 
	$\alpha \in \Br \widetilde{U}_n$. Restricting $\alpha$ to the exceptional
	divisor $E \cong \PP^1_\QQ$, we find that $\alpha$ is constant
	along $E$ as $\Br \PP^1_{\QQ} = \Br \QQ$ (in fact 
	our choice of $\alpha$ is even trivial along $E$).
	Therefore, we could have chosen to instead \emph{define}
	\[U_n(\Adele_\QQ)^{\Br} := \im(\widetilde{U}_n(\Adele_\QQ)^{\Br} 
	\to U_n(\Adele_\QQ)).\]
	as pairing with  $\alpha$ is independent of the choice
	of lift of adelic point from $U_n$ to $\widetilde{U}_n$.
	This is essentially the approach advocated in \cite[\S 8]{CTX13} 
	for dealing with the Brauer--Manin obstruction on singular varieties.
	(Note that in our case the smooth points are dense in $U_n(\QQ_v)$
	for all $v$, so $U_n(\QQ_v) = U_n(\QQ_v)_{\mathrm{cent}}$ in the notation
	of \emph{loc.~cit.})
\end{remark}

\section{Brauer--Manin obstruction} \label{sec:Br}

We now study the integral Brauer--Manin obstruction in our case over $\QQ$. Let $n \in \NN$.

\subsection{Local invariants} 
We begin by calculating the local invariants of the element $\alpha = (-u_1/u_3,-u_2/u_3)$, which we view as an element of $\Br U_n$. We take the convention that the local invariants lie in $\mu_2$, rather than $\ZZ/2\ZZ$. Thus for a place $v$ of $\QQ$ the local invariant map is given by the Hilbert symbol
\begin{equation} \label{eqn:Hilbert}
\inv_v \alpha : U_n(\QQ_v) \to \{\pm 1\}, \quad (u_1,u_2,u_3) \mapsto
(-u_1/u_3, -u_2/u_3)_v.
\end{equation}
The stated expression is only well-defined if $u_1u_2u_3 \neq 0$; for other points, one can reduce to the above case as  the local invariant is continuous \cite[Prop.~8.2.9]{Poo17}.
Indeed, it follows from the implicit function theorem that the $\QQ_v$-points of any dense Zariski-open subset are dense in the smooth points of $U_n(\QQ_v)$; and, as noted in Remark \ref{rem:smooth_points}, the smooth points are dense in $U_n(\QQ_v)$.

\subsection{Real points}

\begin{lemma} \label{lem:real}
	Let
	\[U_n(\RR)_+ = \{ \mathbf{u} \in U_n(\RR) : u_i > 0 \mbox{ for all } i\},
	\quad U_n(\RR)_- = U_n(\RR) \setminus U_n(\RR)_+.\]
	Then the $U_n(\RR)_+$ and $U_n(\RR)_-$ are both connected and
	\[U_n(\RR) = U_n(\RR)_+ \sqcup U_n(\RR)_-, \quad
	\inv_\infty \alpha(U_n(\RR)_+) = -1, \quad \inv_\infty  \alpha(U_n(\RR)_-) = 1.\]
\end{lemma}
\begin{proof}
We first show that $U_n(\RR)$ has two connected components.
Consider 
\begin{equation} \label{eqn:R}
	U_n(\RR) \to \RR^2, \quad \bu \mapsto (u_1,u_2).
\end{equation}
This map is not surjective; indeed, we rearrange the equation \eqref{eqn:awesome} to obtain
\[u_3 = \frac{-u_1u_2}{u_1 + u_2 - 4u_1u_2/n}.\]
So the image misses every point on the hyperbola $u_1 + u_2 - 4u_1u_2/n = 0$, except the origin which is the image of the line $L_{1,2}$. The hyperbola splits the plane into $3$ components, but one branch passes through the origin and hence the image of \eqref{eqn:R} has $2$ components.  The fibres of \eqref{eqn:R} are connected, being a single point or $\RR$ over the origin. Hence $U_n(\RR)$ has $2$ connected components. These are  easily checked to be the two components stated in the lemma. The local invariants are then calculated by standard formulae for Hilbert symbols.
\end{proof}

\subsection{$p$-adic points}

\subsubsection{Preliminaries}

\begin{lemma} \label{lem:p-adic}
	Let $p$ be an odd prime with $v_p(n) \leq 1$
	and $\bu \in \mathcal{U}_n(\ZZ_p)$ with $u_1u_2u_3 \neq 0$.
	Then there exists $i\neq j$ such that $u_i/u_j \in \ZZ_p^*$.
\end{lemma}
\begin{proof}
	Write $u_i = a_ip^{b_i}$ and $n = n' p^{b}$ where $p \nmid n'a_i$.
	The equation \eqref{def:U_n} becomes
	\[4a_1a_2a_3p^{b_1 + b_2 + b_3} = n'(a_1a_2p^{b_1 + b_2 + b} 
	+ a_1a_3p^{b_1 + b_3 + b}
	+ a_2a_3p^{b_2 + b_3 + b}).\]
	Without loss of generality $0 \leq b_1 \leq b_2 \leq b_3$.
	If $b_2 = 0$, then the result is clear. So 
 	assume for a contradiction that $1 \leq b_2 < b_3$. But as $b \leq 1$,
 	we then have
	\[\min\{b_1 + b_2 + b_3, b_1 + b_3 + b, b_2 + b_3 + b\} > b_1 + b_2 + b.\]
	Thus  $p \mid a_1a_2$, which contradicts the fact that the $a_i$
	are units, as required.
\end{proof}

\begin{remark}
Note that Lemma \ref{lem:p-adic} fails in general if $n$ has a prime divisor with valuation at least $2$. For example, for $n= 9$ we have the solution $(4, 6, 36)$.
\end{remark}

\begin{lemma} \label{lem:bad}
	Let $p$ be an odd prime and let $\bu \in \mathcal{U}_n(\ZZ_p)$
	be such that $u_2/u_3 \in \ZZ_p^*$. Then
	\[\inv_p \alpha(\bu) = \left(\frac{-u_2/u_3}{p}\right)^{v_p(u_1u_3)}.\]
\end{lemma}
\begin{proof}
	As $u_2/u_3 \in \ZZ_p^*$, this follows immediately from 
	\eqref{eqn:Hilbert} and standard formulae for Hilbert symbols
	\cite[Thm.~III.1]{Ser73}.
\end{proof}

\subsubsection{Good primes}

\begin{lemma} \label{lem:good}
	For all $p \nmid 2n$ we have
	$\inv_p \alpha(\sU_n(\ZZ_p)) = 1$.
\end{lemma}
\begin{proof}
	
	By continuity, we may assume that $u_1u_2u_3 \neq 0$.
	(The continuity argument above was stated for $\QQ_p$-points, but the $\ZZ_p$-points form an open set in the $\QQ_p$-points so the argument also holds for $\ZZ_p$-points.)
	Up to permuting coordinates, 
	Lemma \ref{lem:p-adic} gives 
	$u_2/u_3 \in \ZZ_p^*$. If $v_p(u_1) = v_p(u_3)$, then the invariant
	is  $1$ by Lemma \ref{lem:bad}. 
	So assume $v_p(u_1) \neq v_p(u_3)$, so that
	$p \mid u_1u_2u_3$. But from the equation \eqref{def:U_n},
	it is clear that $p$ cannot divide only one of the $u_i$
	since $p \nmid 2n$.
	As $u_2/u_3 \in \ZZ_p^*$ we find that $p \mid u_3$.
	From \eqref{eqn:awesome} we have
	\[-\frac{u_2}{u_3} = \frac{1}{1 + u_3/u_1 - 4u_3/n}.\]
	As $p \mid u_3$, $v_p(u_1) \neq v_p(u_3)$,
	and the left hand side is a $p$-adic unit, 
	we must have $v_p(u_3/u_1) > 0$.
	Thus $-u_2/u_3 \equiv 1 \bmod p,$
	and so the local invariant is again trivial 
	by Lemma \ref{lem:bad}.
\end{proof}

\subsubsection{Bad odd primes}

\begin{lemma} \label{lem:surj}
	Let $p \mid n$ be an odd prime. Then the map
	\[\inv_p \alpha : \mathcal{U}_n(\ZZ_p) \to \{\pm 1\}\]
	is surjective.
\end{lemma}
\begin{proof}
	We first consider the case where $p \| n$. Write $n = pn'$
	where $p \nmid n'$ and substitute $u_1 = pa_1$.
	The equation \eqref{def:U_n} becomes
	\[4a_1u_2u_3 = n'(a_1u_2p + a_1u_3p + u_2u_3).\]
	Modulo $p$ this is
	\begin{equation} \label{eqn:planes}
		(4a_1 - n')u_2u_3 \equiv 0 \bmod p. 
	\end{equation}
	As $p$ is odd, there exists a solution with $4a_1 \equiv n' \bmod p$
	and $u_2,u_3$ arbitrary modulo $p$. Geometrically, the equation
	\eqref{eqn:planes} defines the union of three planes which
	is non-singular away from the common points of intersection.
	Providing that $u_2u_3 \not \equiv 0 \bmod p$, we may therefore
	use Hensel's lemma to lift to a $p$-adic solution. Thus, we have
	shown that we may choose $p$-adic solutions such that $p \| u_1,
	p \nmid u_2u_3$ and both possibilities
	\[\left(\frac{-u_2/u_3}{p}\right) = 1,  \quad
	\left(\frac{-u_2/u_3}{p}\right) = -1\]
	may be realised. The result in this case
	therefore follows from Lemma \ref{lem:bad}.
	
	We now consider the general case. Let $n = p^bn'$ where $p \nmid n$
	and $b > 1$.
	We take a $p$-adic solution $\bu \in \sU_{pn'}(\ZZ_p)$ as constructed
	in the previous case, and consider the solution 
	$p^{b-1}\bu\in \sU_{n}(\ZZ_p)$. The quotients 
	$u_1/u_3,u_2/u_3$ are unchanged, hence the result
	follows from the previous case and \eqref{eqn:Hilbert}.
\end{proof}

\subsubsection{The prime $2$}


\begin{lemma} \label{lem:2_even}
Suppose that $n$ is even. Then the map	\[\mathcal{U}_n(\ZZ_2) \to \{\pm 1\}, \quad 
	\bu \mapsto \inv_2 \alpha(\bu)\]
	is surjective.
\end{lemma}
\begin{proof}
	It suffices to prove the result for $n=2$, since then we can just obtain
	the result for all even $n$ by rescaling, as in the proof of Lemma 
	\ref{lem:surj}. Here our equation is
	\[2u_1u_2u_3 = u_1u_2 + u_2u_3 + u_3u_1.\]
	There is the natural number solution $(1,2,2)$ which is easily seen
	to have local invariant $-1$. Next, one verifies
	that the solution 
	\[ (u_1,u_2,u_3) \equiv (-1,2,2) \bmod 8\]
	lifts by Hensel's lemma to a $\ZZ_2$-point with
	local invariant $1$.
\end{proof}

Surprisingly, for odd $n$ the local invariant is always trivial at $2$.

\begin{lemma} \label{lem:2_odd}
Suppose that $n$ is odd.  Then $\inv_2 \alpha(\sU_n(\ZZ_2)) = 1$.
\end{lemma}

\begin{proof}
By continuity it is enough to prove
\[
(-u_1/u_3, -u_2/u_3)_2 = 1
\]
when $u_1u_2u_3 \neq 0$.
Write $\bu = (2^{s_1} r_1, 2^{s_2} r_2, 2^{s_3} r_3)$ with $s_1,s_2,s_3 \ge 0$ and $r_1,r_2,r_3 \in \ZZ_2^\times$. Without loss of generality, we may assume that $s_1 \geq s_2 \geq s_3$.
Looking at valuations in the equation
\[
n(2^{s_1+s_2} r_1 r_2 + 2^{s_1+s_3} r_1 r_3 + 2^{s_2+s_3} r_2 r_3) = 2^{s_1+s_2+s_3+2} r_1 r_2 r_3
\]
shows that $s_1 = s_2$.  Taking out a factor of $2^{s_1+s_3}$ gives
\[
n(2^{s_1-s_3} r_1 r_2 + r_1 r_3 + r_2 r_3) = 2^{s_1+2} r_1 r_2 r_3.
\]
Looking modulo $2$ shows $s_1-s_3 \ge 1$.  We therefore have
\begin{equation}\label{eq:2}
2^{s_1-s_3} r_1 r_2 + (r_1 + r_2) r_3 \equiv 0 \bmod 8.
\end{equation}
Using the formula of~\cite[Thm.~III.1]{Ser73}, the Hilbert symbol above is given by
\[
(-1)^{\epsilon(-r_1/r_3)\epsilon(-r_2/r_3) + (s_1-s_3)(\omega(-r_2/r_3)+\omega(-r_1/r_3))},
\]
where $\epsilon(x) = (x-1)/2$ and $\omega(x) = (x^2-1)/8$.
  Note that $\omega$ is an even function.  We define
\begin{align*}
f(\bu) & = 
\epsilon(-r_1/r_3)\epsilon(-r_2/r_3) = \begin{cases} 1 \bmod 2 & \text{if $r_1 \equiv r_2 \equiv r_3 \bmod 4$} \\
0 \bmod 2 & \text{otherwise}
\end{cases}  \\
g(\bu) & = (s_1-s_3)(\omega(-r_2/r_3)+\omega(-r_1/r_3)) \equiv (s_1-s_3)(\omega(r_1) + \omega(r_2)) \bmod 2.
\end{align*}

If $s_1-s_3 \ge 3$, then~\eqref{eq:2} gives $r_1+r_2 \equiv 0 \bmod 8$, and so $f(\bu)=g(\bu)=0$.  If $s_1-s_3 = 2$, then~\eqref{eq:2} gives $r_1+r_2 \equiv 0 \bmod 4$, and so $f(\bu)=0$; and $g(\bu)=0$ because $s_1-s_3$ is even.

The remaining case is $s_1-s_3=1$.  In this case, \eqref{eq:2} gives $r_1 \equiv r_2 \bmod 4$, which implies
\[
(r_1+r_2)r_3 \equiv -2r_1 r_2 \equiv 6 \bmod 8
\]
and therefore
\[
r_3 \equiv \begin{cases}
1 \bmod{4} & \text{if $r_1+r_2 \equiv 6 \bmod 8$;} \\
3 \bmod{4} & \text{if $r_1+r_2 \equiv 2 \bmod 8$.}
\end{cases}
\]
Now looking at the possible values for $\{r_1,r_2\} \bmod 8$ gives the following.
\begin{center}
\begin{tabular}{cc|ccc}
$r_1 \bmod{8}$ & $r_2 \bmod{8}$ & $r_3 \bmod{4}$ & $f(\bu)$ & $g(\bu)$ \\
\hline
1 & 1 & 3 & 0 & 0 \\
1 & 5 & 1 & 1 & 1 \\
5 & 5 & 3 & 0 & 0 \\
3 & 3 & 1 & 0 & 0 \\
3 & 7 & 3 & 1 & 1 \\
7 & 7 & 1 & 0 & 0
\end{tabular}
\end{center}
Thus, in all cases, $f(\bu) + g(\bu) = 0$, completing the proof.
\end{proof}

\subsection{Proof of Theorem \ref{thm:Hilbert}}
Hilbert's reciprocity law \cite[Thm.~III.3]{Ser73} gives
\[\prod_{p \leq \infty} (-u_1/u_3, -u_2/u_3)_p = 1.\]
For a natural number solution the local invariant at $\infty$ is $-1$ by Lemma~\ref{lem:real}. Moreover, the local invariant at $p \nmid n$ is $1$ by Lemmas \ref{lem:good} and \ref{lem:2_odd}. \qed

\subsection{Proof of Theorem \ref{thm:Hilbert2}}
Similar to the proof of Theorem \ref{thm:Hilbert}, but if one of the $u_i$ is negative then the local invariant at $\infty$ is $1$, by Lemma \ref{lem:real}. \qed

\subsection{Proof of Corollary \ref{cor:Yamamoto}}
The first part of the statement follows from Lemma \ref{lem:p-adic}. For the second part, without
loss of generality we assume that $u_2/u_3 \in \ZZ_p^*$. Then by Theorem \ref{thm:Hilbert} and Lemma \ref{lem:bad}, we deduce that 
\[\left(\frac{-u_2/u_3}{p}\right)^{v_p(u_1u_3)} = -1,\]
whence the Legendre symbol must be $-1$, as required. \qed

\subsection{Proof of Corollary \ref{cor:square}}

By Theorem \ref{thm:Hilbert}, to prove Corollary \ref{cor:square} it suffices to show the following purely local statement (applied to each $p \mid n$).

\begin{lemma}
	Let $p$ be an odd prime and $n=p^{2m}n'$, where $n' \in \ZZ_p^*$
	and $m \geq 0$.
	Let $\bu \in \sU_n(\ZZ_p)$ be such that 
	\[p^{2m} \mid u_1, \ p \nmid u_2u_3, \quad \text{ or }
	\quad p \nmid u_1, \ p^{2m} \mid u_2, \ p^{2m} \mid u_3.\]
	Then $(-u_1/u_3, -u_2/u_3)_p = 1 $.
\end{lemma}
\begin{proof}
	We first consider  type 1 solutions. 
	Here the  Hilbert symbol is
	\[
	\left(\frac{-u_2/u_3}{p}\right)^{v_p(u_1)}.
	\]
	However it follows easily from the equation \eqref{def:U_n}
	that $p \nmid (u_1/n)$, so that $v_p(u_1)$ is even and the result
	follows.

	Now consider type 2 solutions. The equation \eqref{def:U_n}
	implies that $v_p(u_2) = v_p(u_3)$, so the Hilbert symbol
	is
	\[
	\left(\frac{-u_2/u_3}{p}\right)^{v_p(u_3)}.
	\]
	If $p \nmid (u_3/n)$, then $v_p(u_3)$ is even by assumption, and the result
	follows. Otherwise, suppose that $p \mid (u_3/n)$.
	From \eqref{eqn:awesome} we have
	\[-\frac{u_2}{u_3} = \frac{1}{1 + u_3/u_1 - 4u_3/n}
	\equiv 1 \bmod p\]
	since $u_3/u_1 \equiv 4u_3/n \equiv 0 \bmod p$. The
	result follows.
\end{proof}

\subsection{Proof of Theorem \ref{thm:BM}}
First note that as $\sU_n(\ZZ) \neq \emptyset$ and $n >0$, we have 
\begin{equation} \label{eqn:ZZ}
	U_n(\RR)_{+} \times \prod_p \sU_n(\ZZ_p) \neq \emptyset.
\end{equation}
It follows from Lemmas \ref{lem:surj} and \ref{lem:2_even} that there is some prime $p \mid n$ for which the local invariant is surjective on $\sU_n(\ZZ_p)$. Therefore, there are elements of \eqref{eqn:ZZ} whose product of local invariants is $-1$ and $1$, respectively.  \qed

\subsection{Proof of Theorem \ref{thm:Wei}}

The set of real points $U_n(\RR)$ is non-compact. Still, it follows easily from  the equation \eqref{eqn:ES} that 
\begin{equation} \label{eqn:min}
	\min_{\bu \in U_n(\RR)}|u_i| \leq 3n/4.
\end{equation}
These real conditions impose strong arithmetic conditions. (In the terminology of \cite[\S2]{JS17} our surface is ``not weakly obstructed'' but is ``strongly obstructed'' at infinity.) This observation gives the following.

\begin{lemma} \label{lem:Zariski}
	The set
	\[ \{\bu \in \sU_n(\ZZ) : u_1u_2u_3 \neq 0, u_i \neq -u_j \mbox{ for all } i,j \in \{1,2,3\}\} \]
	is finite. In particular, $\sU_n(\ZZ)$ is not Zariski dense and 
	$\sU_n(\NN)$ is finite.
\end{lemma}
\begin{proof}
	Without loss of generality, we have $|u_1|\leq |u_2| \leq |u_3|$.
	Then by \eqref{eqn:min} we have
	$|u_1| \leq 3n/4,$
	so there are only finitely many choices for $u_1$.
	If $4/n = 1/u_1$, then we obtain the solution $u_2 = -u_3$,
	which is being excluded. Hence we have
	\[
	\frac{4}{n} - \frac{1}{u_1}  = \frac{1}{u_2} + \frac{1}{u_3}
	\]
	and the left hand size is non-zero and takes only finitely many
	values. But then as in \eqref{eqn:min}, one finds that 
	$u_2$ and $u_3$ take only finitely many values, as required.
\end{proof}

\begin{lemma} \label{lem:LW}
	For all but finitely many primes $p$, the map
	$\sU_n(\ZZ) \to \sU_n(\FF_p)$
	is not surjective.
\end{lemma}
\begin{proof}
	Follows from Lemma \ref{lem:Zariski} 
	and the Lang--Weil estimates \cite{LW54}
\end{proof}

We now complete the proof of Theorem~\ref{thm:Wei}.
If the map $U_n(\QQ) \to \sU_n(\AA_\QQ)^{\Br}_\bullet$ had dense image then, as $\Br U_n / \Br \QQ$ is finite (Theorem \ref{thm:Brauer}), it would follow from \cite[Lem.~6.5]{CTWX} (applied to $\widetilde{U}_n$) that the map $\sU_n(\ZZ) \to \sU_n(\ZZ_p)$ has dense image for all finitely many primes $p$; however this clearly contradicts Lemma~\ref{lem:LW}, and shows Theorem~\ref{thm:Wei}. \qed

\begin{remark}
Let $X$ be a smooth variety over $\QQ$ which contains a dense torus $T$ with $\HH^0(X_{\bar{\QQ}}, \Gm) = \bar{\QQ}^*$ and $\Pic X_{\bar{\QQ}}$ torsion free. If the action of $T$ on itself extends to $X$, i.e.~$X$ is a toric variety, then in \cite{CX18, Wei14} it is shown that the Brauer--Manin obstruction is the only one to strong approximation away from $\infty$. However this result need not hold if the action of $T$ does not extend to the whole variety.  Here $\widetilde{U}_n$ contains $\Gm^2$ but does not satisfy this result by Theorem~\ref{thm:Wei}.
\end{remark}

\appendix
\section{Comparison with previous results}\label{sec:yamamoto}
In \cite{Yam65} (see also \cite[p.~290]{Mor69}), Yamamoto shows numerous quadratic reciprocity requirements for solutions to~\eqref{eqn:ES} when $n=p$ is prime, with various hypotheses. In this appendix we explain how these are all special cases of Corollary \ref{cor:Yamamoto}.

There are two types of solutions to \eqref{eqn:ES} (see \cite[Ch.~30]{Mor69} and  \cite[Prop.~2.11]{ET13}). Type 1 is when $p$ exactly divides one of the $u_i$ to valuation $1$, and Type 2 is when $p$ divides exactly two of the $u_i$ to valuation $1$.

We first deal with Type $2$ solutions.
Let $\bu \in \sU_p(\NN)$ and suppose that $p \nmid u_1$, $p \| u_2$, $p \| u_3$.  Then one can write (see \cite[p.~289]{Mor69})
\[
(u_1,p^{-1}u_2,p^{-1}u_3) = (bcd,abd,acd)
\]
with $a,b,c,d$ positive integers satisfying $(a,b)=(b,c)=(c,a)=1$, $p \nmid bcd$ and
\[
pa+b+c=4abcd.
\]
Yamamoto~\cite[Lem.~2]{Yam65} defines $q=4abd-1$ and then shows~\cite[Lem.~4]{Yam65} that the Kronecker symbol
\[
\left( \frac{p}{q} \right) = -1.
\]
This follows from Corollary~\ref{cor:Yamamoto}.  Indeed, using $4abd \equiv b/c \bmod p$ we have
\begin{multline*}
\left( \frac{p}{4abd-1} \right)
= \left( \frac{-1}{p} \right) \left( \frac{4abd-1}{p} \right)
= \left( \frac{-b/c}{p} \right)
= \left( \frac{-u_2/u_3}{p} \right) = -1.
\end{multline*}

For Type $1$ solutions, let $\bu \in \sU_p(\NN)$ with $p \mid u_1$, $p \nmid u_2$, $p \nmid u_3$. Write
\[
(p^{-1}u_1,u_2,u_3) = (bcd,acd,abd)
\]
with $a,b,c,d$ positive integers satisfying $(a,b)=(b,c)=(c,d)=1$ and $p \nmid abcd$ (again see \cite[p.~289]{Mor69}).
Then we have
\[
a+bp+cp = 4abcd.
\]
In \cite[Lem.~2]{Yam65} Yamamoto defines $q = 4abd-p$ , assumes $p \equiv 1 \bmod 4$ (see the proof of  \cite[Lem.~4]{Yam65}) and shows that the Kronecker symbol
\[
\left( \frac{p}{4abq} \right) = -1.
\]
By Corollary~\ref{cor:Yamamoto} we deduce this as follows. As $q \equiv 4abd \equiv a/c \bmod p$ we have
\[
\left( \frac{p}{4abq} \right) = \left( \frac{4abq}{p} \right) = \left( \frac{b/c}{p} \right) = \left( \frac{u_3/u_2}{p} \right) = \left( \frac{-u_2/u_3}{p} \right) = -1.\]

Yamamoto also proves two further conditions~\cite[Lem.~3, Lem.~4]{Yam65} in the case $p \equiv 1 \bmod 4$ which, in either the Type $1$ or Type $2$ case, reduce to
\[
\left( \frac{p}{4bc} \right) = -1
\]
where $b,c$ are as defined above for Type $1$ or Type $2$ solutions, respectively.
These also follow from Corollary~\ref{cor:Yamamoto}, as follows:
\[
\left( \frac{p}{4bc} \right) = \left( \frac{4bc}{p} \right) = \left( \frac{b/c}{p} \right) = \left( \frac{u_2/ u_3}{p} \right) = \left( \frac{-u_2/u_3}{p} \right) = -1. 
\]

\bibliographystyle{amsalpha}{}

\begin{thebibliography}{xx}
\bibitem{artin}
M. Artin, On isolated rational singularities of surfaces,
\emph{Amer. J. Math.} \textbf{88} (1966), no.~1, 129--136.

\bibitem{Bri13}
M. Bright, Brauer groups of singular del Pezzo surfaces,
\emph{Michigan Math. J.},
\textbf{62}(3) (2013), 657--664.

\bibitem{CX18}
Y. Cao, F. Xu,
Strong approximation with Brauer-Manin obstruction for toric varieties,
\emph{Ann. Inst. Fourier}, {\bf68}(5) (2018), 1879--1908.

\bibitem{CTS20}
J.-L. Colliot-Th\'el\`ene, A. Skorobogatov, 
\emph{The Brauer--Grothendieck group}, \url{http://wwwf.imperial.ac.uk/~anskor/brauer.pdf}.

\bibitem{CTW12}
J.-L. Colliot-Th\'el\`ene, O.~Wittenberg, Groupe de {B}rauer et points
  entiers de deux familles de surfaces cubiques affines, \emph{Amer. J. Math.}
  \textbf{134} (2012), no.~5, 1303--1327.

\bibitem{CTWX}
J.-L. Colliot-Th\'el\`ene, D.~Wei, F.~Xu, Brauer--Manin obstruction
  for Markoff surfaces, \emph{Annali della Scuola Normale di Pisa}, to appear.

  
\bibitem{CTX13}
J.-L. Colliot-Th\'el\`ene, F.~Xu, Strong approximation for the total space of certain quadric fibrations, \emph{Acta Arith.} \textbf{157} (2013), 169--199.

\bibitem{ET13}
  C. Elsholtz, T. Tao, Counting the number of solutions to the Erd\H{o}s--Straus equation on unit fractions, \emph{J. Aust. Math. Soc.}
  \textbf{94}  (2013), no. 1, 50--105.

\bibitem{grot}
A. Grothendieck, Le groupe de Brauer, III.
In \emph{Dix Expos\'es sur la Cohomologie des Sch\'emas}, North-Holland, 1968.

\bibitem{GS}
P. Gille, T. Szamuely, \emph{Central Simple Algebras and Galois Cohomology},
Cambridge studies in advanced mathematics vol.~101,
Cambridge University Press, Cambridge, 2006.


\bibitem{Har17}
Y. Harpaz, Geometry and arithmetic of certain log K3 surfaces. 
\emph{Ann. Inst. Fourier} {\bf 67} (2017), no. 5, 2167--2200.

\bibitem{Har19}
Y. Harpaz, Integral points on conic log K3 surfaces. 
\emph{J. Eur. Math. Soc.} {\bf 21} (2019), no. 3, 627--664.

\bibitem{JS17}
J. Jahnel, D. Schindler,
On integral points on degree four del Pezzo surfaces,
\emph{Israel J. Math.} {\bf 222} (2017), no. 1, 21--62.

\bibitem{Jar04}
J. H. Jaroma, On expanding $4/n$ into three Egyptian fractions,
\emph{Crux Mathematicorum,} {\bf 30}(1) (2014), 36--37.

\bibitem{LW54} {S.~Lang, A.~Weil},
{Number of points of varieties in finite fields}.
{\em Amer. J. Math.} {\bf76} (1954), 819--827.

\bibitem{lipman}
J. Lipman, Rational singularities, with applications to algebraic surfaces and unique factorization. 
\emph{Inst. Hautes \'{E}tudes Sci. Publ. Math.}
{\bf36}, (1969), 195--279.

\bibitem{LM19}
  D. Loughran, V. Mitankin,
  Integral Hasse principle and strong approximation for Markoff surfaces, 
  \emph{Int. Math. Res. Not.}, to appear.

\bibitem{milne}
J. S. Milne, \emph{\'Etale cohomology.}
Princeton University Press, Princeton, New Jersey, 1980.

\bibitem{Mor69}
	L. J. Mordell, \emph{Diophantine equations.}
	Pure and Applied Mathematics, Vol. 30 Academic Press, London-New York 1969.
	
	
\bibitem{Poo17}
B.~Poonen, \emph{Rational points on varieties}, Graduate Studies in
  Mathematics, vol. 186, American Mathematical Society, Providence, RI, 2017.
  
\bibitem{Ser73}
J.-P.~Serre, \emph{A course in arithmetic}, Graduate Texts in Mathematics, No. 7. Springer-Verlag, 1973.

\bibitem{Sie56}
W. Sierpi\'{n}ski,
  Sur les d\'ecompositions de nombres rationnels en fractions primaires, \emph{Mathesis} {\bf65} (1956), 16--32,.
	
\bibitem{Wei14}
D. Wei, Strong approximation for a toric variety,
\texttt{arXiv:1403.1035}.
  
\bibitem{Yam65}
  K. Yamamoto, 
  On the Diophantine equation $\frac{4}{n}=\frac{1}{x}+\frac{1}{y}+\frac{1}{z}$, \emph{Mem. Fac. Sci. Kyushu Univ. Ser. A}
  \textbf{19} (1965), 37--47.
  
\end{thebibliography}

\end{document}